\documentclass[10pt, notitlepage]{article}
\usepackage{amsmath,amsfonts,latexsym,graphicx, enumerate, amssymb, tikz-cd, amsthm, overpic, longtable}
\usepackage{fullpage,color}
\usepackage{indentfirst}
\usepackage{url,hyperref}
\usepackage[backend=biber, style=numeric, sorting=nty, maxbibnames=99]{biblatex}

\newtheorem{lemma}{Lemma}
\newtheorem{theorem}[lemma]{Theorem}
\newtheorem{defn}[lemma]{Definition}

\newtheorem{corollary}[lemma]{Corollary}
\newtheorem{prop}[lemma]{Proposition}

\newtheorem{remark}[lemma]{Remark}

\newtheorem{question}{Question}
\newtheorem{example}{Example}
{\medskip}

\makeatletter
\renewenvironment{proof}[1][\proofname]{%
      \par\pushQED{\qed}\normalfont%
      \topsep6\p@\@plus6\p@\relax    
      \trivlist\item[\hskip\labelsep\bfseries#1\@addpunct{.}]%
      \ignorespaces 
}{%
      \popQED\endtrivlist\@endpefalse 
}

\date{}
\title{Right-angled links in thickened surfaces}
\author{Rose Kaplan-Kelly}

\begin{document}
\maketitle

\begin{abstract}
    Traditionally, alternating links
    are studied with alternating diagrams on $S^2$ in $S^3$. In this paper, we consider links which are alternating on higher genus surfaces $S_g$ in $S_g \times I$. We define what it means for such a link to be right-angled generalized completely realizable (RGCR) and show that this property is equivalent to the link having two totally geodesic checkerboard surfaces, equivalent to each checkerboard surface consisting of one type of polygon, and equivalent to a set of restrictions on the link's alternating projection diagram. We then use these diagram restrictions to classify RGCR links according to the polygons in their checkerboard surfaces, provide a bound on the number of RGCR links for a given surface of genus $g$, and find an RGCR knot. Along the way, we answer a question posed by Champanerkar, Kofman, and Purcell about links with alternating projections on the torus.
\end{abstract}

\section{Introduction}

We traditionally consider a link, $L$, with its projection diagram on $S^2$ and its complement in $S^3$. We say that such a link, $L$, is \textit{hyperbolic} if its complement admits a complete hyperbolic structure. An overarching question about hyperbolic links, and in knot theory more generally, is how one can gather information about the geometric properties of a link complement from the link's diagrams. Alternating links have proven especially useful for this question because of the amount of information we can gain about an alternating link complement in $S^3$ from its alternating diagram. Menasco showed that all prime nonsplittable alternating links are hyperbolic except for $(2, q)-$ torus links \cite{MenascoClassic}. Howie \cite{HowieAlt} and Greene \cite{GreeneAlt} independently showed that alternating links can be characterized by their checkerboard surfaces. Using the checkerboard surfaces of the link's alternating diagram, Menasco described the polyhedral decomposition of an alternating link complement \cite{Menasco}. It follows that we would like to consider broader classes of links which share these helpful properties and see which techniques for studying alternating link complements in $S^3$ can be generalized.

 In this paper, we will consider links with alternating projections on higher genus surfaces. These links have been studied by several authors. Adams considered toroidally alternating links, a generalization of alternating links to projections on a Heegaard torus \cite{AdamsAlone}. Hayashi \cite{Hayashi} and Ozawa \cite{Ozawa} studied links with projections on higher genus surfaces. 
 Adams, Calderon, and Mayer considered links in thickened surfaces constructed from $k$-uniform tilings of $\mathbb{H}^2, \mathbb{S}^2$ and $\mathbb{E}^2$ \cite{ACM}. Adams, Albors-Riera, Haddock, Li, Nishida, Reinoso, and Wang showed that links satisfying a generalized alternating and prime condition in thickened surfaces are hyperbolic \cite{Adams.etal}. Champanerkar, Kofman, and Purcell considered semi-regular links, alternating links in the thickened torus which correspond to convex regular Euclidean tilings. Howie introduced weakly generalized alternating links and showed that they are nontrivial, nonsplit, and prime \cite{HowieThesis}. Howie and Purcell introduced broader generalizations of alternating links in compact $3$-manifolds. They found an angled chunk decomposition for these generalized link complements and proved a more general result providing conditions for their complements to be hyperbolic \cite{HowiePurcell}. 
 
 Throughout this paper we will assume that our projection surface, $F$, is a connected, closed, orientable, surface with non-positive Euler characteristic. We will assume our links to be in the thickened surface $F \times I$ for $I=[-1,1]$ with generalized projection diagram $\pi(L)$ on $F \times \{0\}$ and specify the genus of the surface as needed. We will use the following definitions adapted from \cite{HowiePurcell} for this setting. 

\begin{defn}
\normalfont
    The generalized diagram $\pi(L)$ of a link $L$ on surface $F$ with genus $g \geq 1$ is \textit{weakly prime} if for any disk $D$ in $F$ with $\partial D$ that intersects the link diagram transversely exactly twice, $\pi(L) \cap D$ is an embedded arc.
\end{defn}

\begin{defn}
\normalfont The diagram $\pi(L)$ is \textit{cellular} if all complementary regions, called faces, on $F$ are disks. 
\end{defn}

\begin{defn}
\normalfont A generalized diagram is \textit{alternating} if the boundary of each region of $F \backslash \pi(L)$ admits an orientation such that crossings alternate between under and over as described below.
\end{defn}

Orient each region of $F \backslash \pi(L)$ such that the induced orientation on the boundary runs from under-crossings to over-crossings. Next, color the faces white or shaded depending on whether their orientation is clockwise or counter-clockwise. The result of this coloring will be that faces which share an edge have opposite orientations, and therefore are different colors, while faces that are diagonal from each other with respect to a crossing have the same color. This coloring is called a \textit{checkerboard coloring} of the generalized link diagram. A link with a generalized diagram on $F$ that is weakly prime, cellular, and alternating is also checkerboard colorable.

\begin{defn}
\normalfont Given a generalized diagram $\pi(L)$ with a checkerboard coloring, we can form two \textit{checkerboard surfaces}. Insert a twisted band at each crossing of $(F \times I) \backslash L$ to connect all shaded faces. The result is a surface composed of shaded faces with boundary the link. Perform the same procedure for the white faces. The result is two surfaces that intersect in arcs in $(F \times I) \backslash L$ at each crossing. These \textit{crossing arcs} run between the under-strand and over-strand of $L$.
\end{defn}

\begin{defn}
\normalfont A generalized diagram $\pi(L)$ is \textit{reduced alternating} if it is alternating, weakly prime, and every component of $L$ projects to at least one crossing of $\pi(L)$. 
\end{defn}

Note that if our link diagram is alternating, weakly prime, and cellular then it is reduced alternating as well but reduced alternating diagrams need not be cellular.

Howie and Purcell call links that are reduced alternating and checkerboard colorable \textit{weakly generalized alternating}, with the additional criteria that the links satisfy a so-called representativity condition \cite{HowiePurcell}. This condition will always hold for our setting in thickened surfaces because $F$ has no compression disks in $F \times I$. 

The main results of this paper are the following:

\begin{theorem}
\label{TFAE1}
Let $L$ be a cellular weakly generalized alternating link on $F$ in $F \times I$ with reduced alternating diagram $\pi(L)$. Then the following are equivalent:
\begin{enumerate}
    \item The complete hyperbolic structure on $(F \times I) \backslash L$ with totally geodesic boundary is the result of gluing two right-angled generalized polyhedra built from the checkerboard surfaces of $\pi(L)$ by rotation on their faces.
    \item $L$ has two totally geodesic checkerboard surfaces.
    \item Each checkerboard surface of $\pi(L)$ contains polygons of precisely one number of sides.
    \item At most two types of polygons occur in $\pi(L)$. Either $\pi(L)$ consists entirely of one type of regular polygon or both types, one with $n$ sides and the other with $m$ sides, are regular, and about each vertex the polygons are arranged in an $[n, m, n, m]$ pattern. 
\end{enumerate}
\end{theorem}

We say that a link is \textit{right-angled generalized completely realizable (RGCR)} if it satisfies the equivalent conditions of Theorem \ref{TFAE1}. 

Gan showed that there are precisely three links in $S^3$ with complements that admit this type of right-angled structure \cite{Gan} and Champanerkar, Kofman, and Purcell showed that there are only two convex regular Euclidean tilings that correspond to links in the thickened torus with this structure \cite{CKPBiperiodic}. In this paper we show:

\begin{theorem}
    Given a projection surface $F$ of genus $g$, there are finitely many tilings of $\mathbb{H}^2$ and $\mathbb{E}^2$ which correspond to RGCR links, $L$, with $\pi(L)$ on $F$. Moreover, for each $F$ of genus $g > 1$, the number of RGCR links is bounded above by $\displaystyle{\bigg(\frac{310g^2}{9}-\frac{101g}{3}+4\bigg)[(84g-83)!]}$.
    \end{theorem}
While there are many ways in which the geometry of the complements of hyperbolic links in thickened surfaces is similar to that of complements in $S^3$, there are also ways in which in differs. Champanerkar, Kofman, and Purcell conjecture that there are no knots in the $3$-sphere with right-angled structure on their complements \cite[Conjecture 5.12]{CKPright}. In Section $5$ we show:
    \begin{theorem}
There exists a right-angled knot in thickened surfaces.
\end{theorem}
 
 Champanerkar, Kofman, and Purcell showed that semi-regular links admit a decomposition into an appropriately generalized version of ideal polyhedra. These generalized ideal polyhedra, which they call torihedra, glue to give the complete hyperbolic structure on the link complement \cite{CKPBiperiodic}. This generalized Aitchison and Reeves's notion of completely realizable links \cite{AR}. They further proved that there are only two semi-regular links, the square weave and the triaxial link, which admit their torihedral decomposition with right-angled torihedra \cite{CKPBiperiodic}. They ask:
 \begin{question}
 \normalfont
 (Champanerkar-Kofman-Purcell, \cite[Question 5.2]{CKPBiperiodic})
 Besides the square weave $W$ and the triaxial link $L$, do there exist any other right-angled biperiodic alternating links?
 \end{question}
 
We answer this question with Corollary \ref{TorusCor} by showing as a corollary to Theorem \ref{TFAE1} that the triaxial link and square weave are the only biperiodic alternating links with corresponding RGCR quotient links on the torus. Champanerkar, Kofman, and Purcell used this right-angled structure to prove that both the Volume Density Conjecture and the Toroidal Vol-Det Conjecture hold with equality for the square weave and the triaxial link \cite[Theorem 6.3 and Theorem 6.7]{CKPBiperiodic}. 

\subsection{Organization} Section $2$ outlines two topological decompositions of the complement of cellular weakly generalized alternating link given by \cite{HowiePurcell} and \cite{ACM}. In Section $3$ we consider the complete hyperbolic structure on these complements and show that if a cellular weakly generalized alternating link has two totally geodesic checkerboard surfaces then it is completely realizable and the faces of its hyperbolic generalized polyhedra are regular. 
Section $4$ proves Theorem \ref{TFAE1}.
In Section $5$ we prove that there is an upper bound on the number of RGCR links for each projection surface with negative Euler characteristic and construct an RGCR knot. We conclude with a table listing all of the tilings of $\mathbb{H}^2$ that correspond to RGCR links with projections on surfaces of genus 2 through 4.

\subsection{Acknowledgements}
I would like thank my advisor, Dave Futer, for his guidance and advice.
I would also like to thank Colin Adams and Ilya Kofman for their helpful comments. This research was partially supported by National Science Foundation grant DMS-1907708.

\section{Topological Decompositions}

In this section we will describe the topological generalized polyhedral decomposition of a cellular weakly generalized alternating link complement. For projections on a torus, Champanerkar, Kofman, and Purcell constructed a decomposition of the link complement into two ideal torihedra (defined below) and a decomposition into ideal tetrahedra \cite{CKPBiperiodic}. Adams, Calderon, and Mayer found a bipyramid decomposition of generalized alternating links in thickened surfaces including higher genus \cite{ACM}. Considering even broader classes of generalized alternating link complements, Howie and Purcell found a chunk decomposition \cite{HowiePurcell}. In this paper, we will decompose the link complements into generalized semi-truncated polyhedra as defined below. 

\begin{defn}
\normalfont A \textit{generalized polyhedron} is homeomorphic to $(F \times [0,1])/(F \times \{1\})$ with a cellular graph $\Gamma$ on $F \times \{0\}$. An \textit{(ultra) ideal generalized polyhedron} is a generalized polyhedron with the vertices of $\Gamma$ and the vertex at $F \times \{1\}$ removed. When $F$ has genus $1$, all vertices are ideal. When $F$ has genus greater than $1$, the vertex corresponding to $F \times \{1\}$ is outside of the boundary at infinity (ultra-ideal) while the vertices removed from $\Gamma$ are ideal. When there is an ultra-ideal vertex, the generalized polyhedron is \textit{semi-truncated} meaning that the ideal vertices are removed and a neighborhood of the ultra-ideal vertex is removed.
\end{defn}

Observe that if $F$ is a torus, then the generalized polyhedron is an \textit{ideal torihedron} as in \cite{CKPBiperiodic}. For ease of reading we will usually refer to `generalized polyhedra' to indicate any of these types of generalized polyhedra and further specify when the vertex types become pertinent. The following is a restriction of the decomposition given by \cite[Propositions 3.1 and 3.3]{HowiePurcell} which in turn generalizes Menasco's polyhedral decomposition \cite{Menasco}. 

\begin{prop}(Howie-Purcell, \cite[Propositions 3.1 and 3.3]{HowiePurcell}).
\label{Decomp}
Let $F$ be a surface of genus greater than $0$ and $L$ be a cellular weakly generalized alternating link in $F \times I$ with projection diagram $\pi(L)$. Then $(F \times I) \backslash L$ decomposes into two ideal torihedra or two semi-truncated generalized polyhedra, $P^+$ and $P^-$, satisfying:
\begin{itemize}
    \item $P^+$ and $P^-$ are homeomorphic to $F \times [0,1)$ and $F \times (-1,0]$ or $F \times [0,1]$ and $F \times [-1,0]$ with a finite set of points removed from each $F \times \{0\}$ boundary.
    
    \item On each copy of $F$ there is an embedded graph in which vertices, edges, and regions correspond to $\pi(L)$. The vertices are ideal and 4-valent.

    \item We can color the regions of the faces $F \backslash \pi(L)$ on $P^+$ and $P^-$ shaded and white to correspond to the checkerboard coloring of $\pi(L)$.
    
    \item $(F \times I)\backslash L$ is the result of gluing $P^+$ to $P^-$ by a homeomorphism between the copies of $F \backslash \pi(L)$ on their boundaries. This homeomorphism is the composition of the identity with a rotation on each face of $F \backslash \pi(L)$. One can choose an orientation such that white faces are rotated to their neighboring edge clockwise, while shaded faces are rotated to their neighboring edge counterclockwise. 
    
    \item Edges correspond to crossing arcs and are glued in collections of four. Each ideal vertex has pairs of opposite edges which are glued together.
    
\end{itemize}
\end{prop}

A generalized polyhedral decomposition for link complements in thickened surfaces also appears in the work of Adams, Calderon, and Mayer, with the additional step of subdividing the generalized polyhedra into ideal or truncated $n$-bipyramids \cite[Lemma 2.10]{ACM}. 

\begin{theorem}(Adams--Calderon--Mayer, \cite[Lemma 2.10 and Theorem 4.4]{ACM}).
Let $L$ be a link in $F \times I$ corresponding to a tiling $T$ by regular polygons of $S^2$, $\mathbb{E}^2$, or $\mathbb{H}^2$. Let $\{F_i\}$ be the collection of $m$ non-bigon faces complementary to the projection of $L$ to $F$ and let $n_i$ be the number of edges in the $i$-th
face. Then: 
\begin{enumerate} 
\item Topologically, $(F \times I) \backslash L$ decomposes into $m$ face-centered bipyramids, each corresponding to
a unique face with $n_i$ edges.
\item The complete hyperbolic structure on $(F \times I) \backslash L$ (with totally geodesic boundary when $F$ has negative Euler characteristic) is obtained by gluing together symmetric hyperbolic bipyramids, each corresponding to a face of the tiling of $F$.
\end{enumerate}
\end{theorem}

\section{Totally Geodesic Checkerboard Surfaces}

In this section we consider cellular weakly generalized alternating links with two totally geodesic checkerboard surfaces and the geometric structure on their complements. When our projection surface has negative Euler characteristic, $(F \times I) \backslash L$ has many different hyperbolic structures. Throughout this paper we choose to study the unique hyperbolic structure where the boundary components of the complement corresponding to $F \times \{-1\}$ and $F \times \{1\}$ are totally geodesic. For further information about this choice of structure see \cite{HowiePurcell} and \cite{ACM}.
We begin with a few more definitions. 

\begin{defn}
\normalfont
We say that a generalized polyhedron is \textit{hyperbolic} if it admits a convex hyperbolic structure with totally geodesic boundary.
\end{defn}

\begin{defn}
\normalfont If we can chose an orientation on two hyperbolic generalized polyhedra, $P^+$ and $P^-$, such that each pair of corresponding $n$-gon faces on $P^+$ and $P^-$ can be glued by a $1/n$ rotation either clockwise or counterclockwise depending on the checkerboard surface they belong to, then we call this a gluing by \textit{gear shift rotation}. 
\end{defn}
 
\begin{defn}
\normalfont A link $L$ is \textit{completely realizable} if the complete hyperbolic structure on $(F \times I) \backslash L$ is given by gluing hyperbolic generalized polyhedra by gear shift rotation and these hyperbolic generalized polyhedra have the combinatorial structure of the generalized checkerboard polyhedra of $\pi(L)$.  
If, in addition, all dihedral angles in both generalized checkerboard polyhedra are $\pi/2$ then we say that the link is \textit{right-angled generalized completely realizable} (RGCR).
\end{defn}

Note that this definition is a generalization of right-angled completely realizable links \cite[Definition 3.13]{Gan}.

Howie and Purcell showed that a cellular weakly generalized alternating link with projection surface $F \times \{0\}$ in $F \times I$ is hyperbolic \cite[Theorem 4.2]{HowiePurcell}. Adams, Calderon, and Mayer showed that for such a link with a diagram consisting of regular polygons, their $n$-bipyramid decomposition can be realized as a geometric decomposition \cite[Theorem 4.4]{ACM}.
The links of the ideal or ultra-ideal vertices corresponding to the apexes of these bipyramids are regular polygons with interior angles determined by the interior angles of the polygons on the projection surface \cite[Lemma 4.3]{ACM}, \cite[Corollary 3.10]{CKPBiperiodic}. We can see this by noting that if $F$ has hyperbolic structure, then the apexes of the bipyramids are ultra-ideal and the truncation plane is the unique hyperbolic plane perpendicular to the collection of edges of the bipyramids with endpoint at that apex. If $F$ is a torus, then $F \times \{1\}$ and $F \times \{-1\}$ get their Euclidean structures from the truncated faces of bipyramids with ideal apexes. 
Recalling that for a fixed interior angle there is only one regular hyperbolic polygon up to isometry and one regular Euclidean polygon up to similarity, and noting that both $F \times \{-1\}$ and $F \times \{1\}$ are the result of gluing these regular boundary polygons, we see that the boundary surfaces have the same geometric structure as the projection surface.

\begin{lemma}
\label{CRandIsoLemma}
Let $L$ be a cellular weakly generalized alternating link. If $\pi(L)$ has two totally geodesic checkerboard surfaces then $L$ is completely realizable and the hyperbolic generalized checkerboard polyhedra associated to $\pi(L)$ are isometric.
\end{lemma}

\begin{proof}
 We claim that there is a hyperbolic generalized polyhedron with the combinatorial structure of the generalized checkerboard polyhedra associated to $\pi(L)$. First, consider the generalized checkerboard polyhedron $P^+$. 
 Lift $P^+$ to $\mathbb{H}^3$ and consider a fundamental domain, $\widehat{P^+}$. There is a fixed angle between the link diagram's totally geodesic checkerboard surfaces because each checkerboard surface lifts to parallel totally geodesic planes. Thus there is a fixed angle between the faces of the generalized polyhedron $P^+$. Since both checkerboard surfaces are totally geodesic, we find that any region $S$, consisting of an $n$-gon and its interior in one of the checkerboard surfaces, lifts to an ideal $n$-gon $\widehat{S}$ in $\widehat{P^+}$. 
 The ideal vertices of this lift are distinct (\cite[Proposition 2.1]{TT}), which implies that none of the boundary faces of $P^+$ collapse in the pre-image.
 
 The other generalized checkerboard polyhedron associated to $\pi(L)$, $P^-$, has the same dihedral angles between its totally geodesic faces as $P^+$.
 Additionally, the surface boundary components of both generalized polyhedra are incompressible because both copies of $F$ have no compression disks in $F \times I$. We can therefore apply \cite[Theorem 8.15]{Schlenker} to see that $P^+$ and $P^-$ must have the same complete hyperbolic structure with totally geodesic surface boundary. This implies that there is a hyperbolic generalized polyhedron with the combinatorial structure of $P^-$ and this realization of $P^-$, $\widehat{P^-}$, is the mirror image of $\widehat{P^+}$. 
Both hyperbolic generalized polyhedra have the combinatorial structure of the topological checkerboard polyhedra and therefore glue by gear shift rotation on their totally geodesic faces to give the complete hyperbolic structure on the link complement.
 \end{proof}

Observe that if both hyperbolic generalized polyhedra are right-angled then there is a quick way in which to show that they are isometric \cite[Lemma 8.14]{Schlenker}. Take $P^+$ and double it along first its shaded checkerboard surface and then along its white checkerboard surface. The result is a finite volume hyperbolic manifold, $D_wD_s(P^+)$, with totally geodesic surface boundary and cusps. 
By Mostow-Prasad rigidity $D_wD_s(P^+)$ has a unique hyperbolic structure.
Using that both checkerboard surfaces are totally geodesic, we see that reflecting in either checkerboard surface in $D_wD_s(P^+)$ is an isometry. Therefore we can construct an isometry between $P^+$ and $P^-$ by reflecting on both checkerboard surfaces. 

We define a regular polygon as in \cite{AR}. 
\begin{defn}
\normalfont Let $T$ be a convex, planar, ideal, hyperbolic polygon with $n$ sides. Then $T$ is \textit{regular} if it is set-wise invariant under a rotation of order $n$. 
\end{defn}

\begin{theorem}
\label{RegularityThm}
Let $L$ be a cellular weakly generalized alternating link with generalized checkerboard polyhedra $P^+$ and $P^-$. If $L$ has two totally geodesic checkerboard surfaces then the faces of $P^+$ and $P^-$ are regular.
\end{theorem}

\begin{proof}
 By Lemma \ref{CRandIsoLemma}, $L$ is completely realizable. Therefore the complete hyperbolic structure on $(F \times I) \backslash L$ is given by gluing two generalized checkerboard polyhedra, $P^+$ and $P^-$, by gear shift rotation. Let $T_i$ be the totally geodesic faces of $P^+$ and $T_i'$ be the faces of $P^-$. 
   
   We follow Gan's proof \cite[Theorem 3.14]{Gan}. The gear shift rotation provides us with a collection of gluing isometries $\{\psi_i: T_i \rightarrow T_i'\}$. These isometries, which are rotations, glue the faces without shearing on any edge because $L$ is completely realizable.
   
   By Lemma \ref{CRandIsoLemma}, the generalized checkerboard polyhedra $P^+$ and $P^-$ are isometric. Consider the restriction of the isometry, $\phi: P^- \rightarrow P^+$, to the copy of $F \backslash \pi(L)$ on the boundary of $P^-$. We find our isometry by reflecting across the totally geodesic checkerboard surfaces whose boundaries are the link diagram, so $\phi$ will fix $\pi(L)$. 
   This implies that the two copies of $F \backslash \pi(L)$ have the same hyperbolic structure induced by the hyperbolic structure on $P^+$ and $P^-$.  
   
   Now consider the collection of maps $\phi \circ \psi_i: T_i \rightarrow T_i$. First, $\psi_i$ rotates and glues $T_i$ to $T_i'$, then $\phi$ sends $T_i'$ back to $T_i$. The result is an isometry of each totally geodesic face $T_i$ which is an order $n$ rotation.
\end{proof}

The following is a slight generalization of \cite[Lemma 3.9]{Gan}. We follow the proof from \cite{Gan} adjusted for this setting.

\begin{prop}
\label{Gan_DiagsEqu}
Let $L$ be a cellular weakly generalized alternating link with diagram $\pi(L)$. Suppose that the two checkerboard surfaces of $\pi(L)$ are totally geodesic in $(F \times I) \backslash L$. If $G_n$, $G_m$ are an $n$-gon and $m$-gon in the diagram that are diagonal from each other, then $n=m$. In other words, polygons in the same totally geodesic checkerboard surface have the same number of sides.
\end{prop}

\begin{proof}
 By Theorem \ref{RegularityThm} the polygons in the checkerboard surfaces are regular. By assumption $G_n$ and $G_m$ are diagonal from each other with respect to a vertex in $\pi(L)$ so they are part of the same totally geodesic checkerboard surface. Suppose that this is the white checkerboard surface $\Sigma_W$. Lift both faces to $\mathbb{H}^3$. The result will be totally geodesic polygons which are adjacent up to translations by the lifts of $\Sigma_W$. Consider the pre-image such that the faces share an edge with vertices at $0$ and $\infty$, as shown in Figure \ref{fig:LiftedPolygons}.

\begin{figure}[ht]
\centering
\begin{overpic}[abs,unit=1mm,scale=.5, trim=0 8.25in 0 0, clip]
    {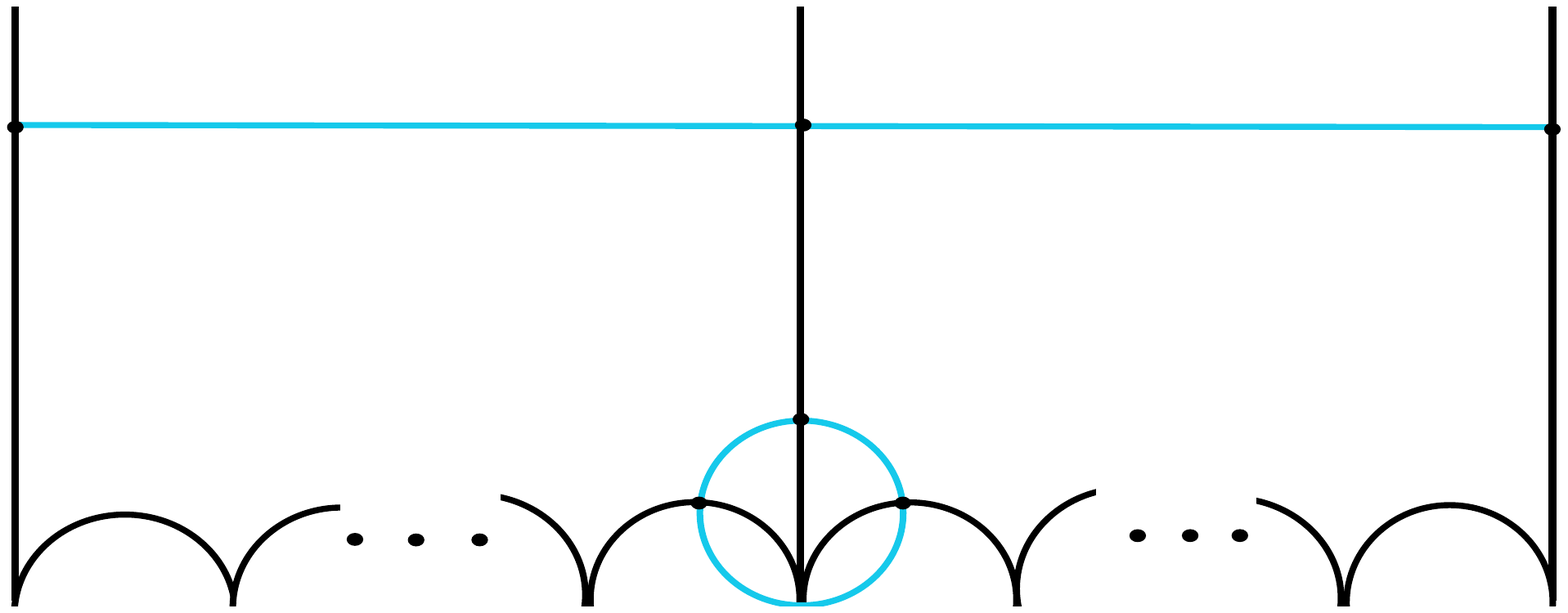}
    
    \put(7,36){A}
    \put(55, 36){Q}
    \put(95, 36){A'}
    \put(44, 12){B}
    \put(61, 12){B'}
    \put(50,18){P}
    \put(4,0){z}
    \put(53,0){0}
    \put(99,0){-z}
    \put(39,0){w}
    \put(65,0){-w}
    
    \end{overpic}
    
    \caption{Figure modified from \cite{Gan}. Lifts of $G_n$ and $G_m$ in $\widehat{\Sigma_W}$ are shown as a vertical plane in the upper half-space model. The intersection of $\widehat{\Sigma_W}$ with the boundaries of horoball lifts of a neighborhood of $L$ appear in blue and the polygons appear in black.}
    
    \label{fig:LiftedPolygons}
    
\end{figure}

Choose the cusp sizes such that all of their boundaries have a meridian of the same length in the Euclidean metric induced from the complete hyperbolic structure on the link complement. The link diagram's two totally geodesic checkerboard surfaces give the boundary of each cusp a tiling by quadrilaterals with interior angle the angle between the two totally geodesic checkerboard surfaces of $\pi(L)$. One such tiling is shown in Figure \ref{fig:quadrangulation}.

Using that this angle is fixed and that the totally geodesic surfaces are embedded, we see that the quadrilaterals have parallel sides. Consider one cusp. One of the diagonals of all of the parallelograms on this cusp's boundary corresponds to a meridian on the boundary of a tubular neighborhood of the link component corresponding to this cusp. Therefore one diagonal of each parallelogram is the same length. So, adjacent parallelograms have sides of the same (respective) lengths. This is shown in Figures \ref{fig:quadrangulation} and \ref{fig:harlequin}.

\begin{figure}
    \centering
    \includegraphics[width=3.75in, trim= {0 6.25in 0 0}]{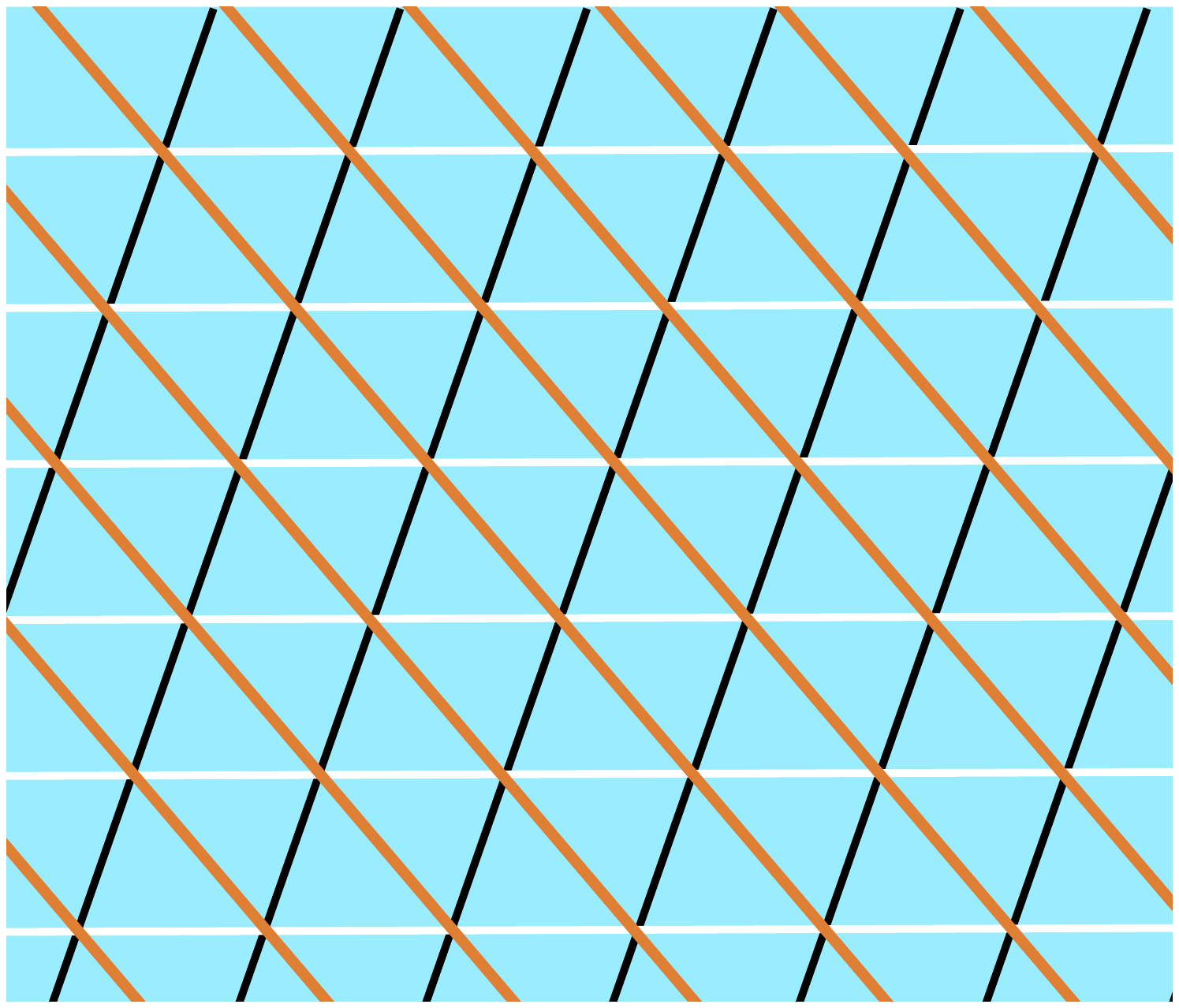}
    \caption{A tiling of the boundary of a cusp by its intersection with the white and shaded totally geodesic checkerboard surfaces. The cusp is shown in blue, the intersection with lifts of the white checkerboard surface are shown in white, and the intersection with lifts of the shaded checkerboard surface are shown in black. The pre-image of the meridian, the diagonal of each parallelogram that is the same fixed length, is shown in orange.}
    \label{fig:quadrangulation}
\end{figure}

\begin{figure}
    \centering
    \begin{overpic}[width= 4in, trim={0 7.5in 0 0in}]{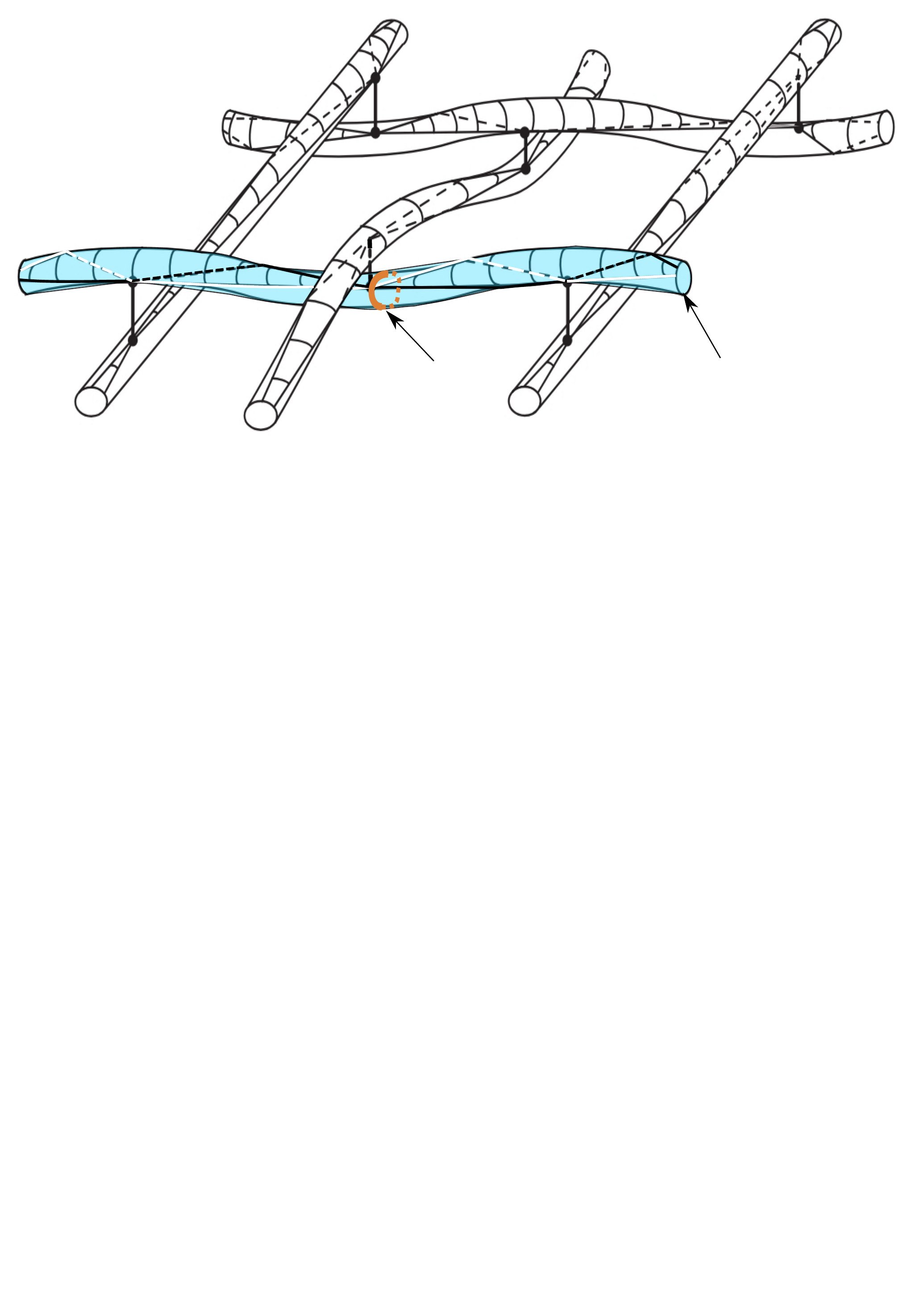}
    \put(41,8){meridian}
    \put(73,9){cusp torus}
    \end{overpic}
   \caption{Figure modified from \cite{harlequin} and \cite{Gan}. Tiling of the boundary of cusps of $(F \times I) \backslash L$ by quadrilaterals. A meridian is shown in orange.}
    \label{fig:harlequin}
\end{figure}

In Figure \ref{fig:LiftedPolygons} we see that the sides $AQ$ and $QA'$ are the same length, as are $BP$ and $PB'$. We have the following consecutive vertices of $\widehat{G_n}$: $z, \infty, 0, w$ where $z$ and $w$ are complex. Using that the segments are of equal length, we see that $-z, \infty, 0,$ and $-w$ are the corresponding symmetric consecutive vertices of $\widehat{G_m}$. So, $\widehat{G_n}$ and $\widehat{G_m}$ have consecutive vertices with the same cross ratio. Aitchison and Reeves showed that a hyperbolic polygon with $x$ sides is regular if and only if the cross ratio of any four consecutive vertices of the polygon is $1+ 1/ (2 \cos(2 \pi /x)+1)$ \cite[Lemma 3.2]{AR}. Our two sets of vertices have the same cross ratio so our two lifted polygons then have the same number of sides. 
\end{proof}

\section{Right-angled Structure and Totally Geodesic Checkerboard Surfaces}

In this section we consider links whose complements admit a complete hyperbolic structure formed by gluing generalized checkerboard polyhedra with dihedral angles of $\pi/2$.
\begin{defn}
\normalfont
 We say that tiling $T'$ of $\mathbb{H}^2$ is \textit{equivalent} to tiling $T$ if there exists a homeomorphism of $\mathbb{H}^2$ such that vertices, edges, and faces of $T'$ are sent to $T$.
\end{defn}

\begin{theorem}
\label{TFAE}
    Let $L$ be a cellular weakly generalized alternating link on $F$ in $F \times I$. Then the following are equivalent:
    
    \begin{enumerate}
        \item $L$ is RGCR.
        \item $L$ has two totally geodesic checkerboard surfaces.
        \item The checkerboard surfaces of $\pi(L)$ each have exactly one type of polygon.
        \item $L$ has a projection diagram $\pi(L)$ on $F$ with at most two types of polygons, one with $n$ sides and the other with $m$ sides, such that the polygons are arranged in an $[n, m, n, m]$ pattern about each vertex and are regular.
    \end{enumerate}
    
\end{theorem}

\begin{proof}
$(1)\implies(2)$
By assumption the complete hyperbolic structure on the link complement is given by gluing $P^+$ and $P^-$ by gear shift rotation. Consider the totally geodesic faces of these hyperbolic generalized polyhedra. If we glue just the shaded faces of $P^+$ and $P^-$ then we see that the white faces meet in pairs at a $\pi$ angle which is the sum of the two $\pi/2$ dihedral angles between white and shaded faces on each hyperbolic generalized polyhedron. This also holds if we glue only the white faces. Therefore the checkerboard surfaces of $\pi(L)$ are both totally geodesic and the hyperbolic structure on each checkerboard surface is determined by the hyperbolic structure of the generalized polyhedra. 

$(2) \implies (3)$
Follows from Proposition \ref{Gan_DiagsEqu}.

$(3) \implies (4)$
The checkerboard surfaces of $\pi(L)$ each contain one type of polygon, therefore $\pi(L)$ must have faces consisting of polygons of $n$ numbers of sides and $m$ numbers of sides in a pattern $[n,m,n,m]$ about each vertex. If the faces are regular then we are done. 

Suppose that the polygons are not regular. If $F$ has negative Euler characteristic then the link diagram on $F$ in $F \times I$ corresponds to a quotient of a tiling of $\mathbb{H}^2$ by $n$-gons and $m$-gons arranged in the specified pattern. Call this tiling $T$. This $[n, m, n, m]$ pattern about each vertex of the tiling satisfies condition $(1)$ in \cite{DattaandGupta}. Therefore there exists a semi-regular tiling $T'$ of $\mathbb{H}^2$ by regular polygons with vertex type $[n,m,n,m]$ at every vertex, which is equivalent to $T$ \cite[Lemma 2.5]{DattaandGupta}.
Using this regular tiling of $\mathbb{H}^2$ we find the diagram of $L$ with regular polygons on $F$. If $F=T^2$ we apply the same argument using \cite[Theorem 1.3]{DattaandMaity} to find $\pi(L)$ as the quotient of a regular tiling of $\mathbb{E}^2$. 

$(4) \implies (1)$
 By assumption, we can construct $L$ as a quotient of a regular tiling of $\mathbb{E}^2$ or $\mathbb{H}^2$ with one vertex type, so it is a uniform tiling link (in the terminology of \cite{ACM}). Adams, Calderon, and Mayer decomposed the complements of these links into a collection of symmetric bipyramids \cite[Theorem 4.4]{ACM}. They further found the angles of these bipyramids in terms of the angles of the polygons in $F \backslash \pi(L)$ by cutting each bipyramid into a collection of wedges with angles $D= \pi - \alpha$, $B=C=E=F= \alpha/2$, and $A=2\pi/n$ for each $n$-gon in the projection diagram with interior angle $\alpha$ (see Figure \ref{fig:wedges} for labels).
    \begin{figure}
        \centering
        \begin{overpic}[abs,unit=1mm,scale=.5, trim={0in 6.75in 0in 0in}, clip]{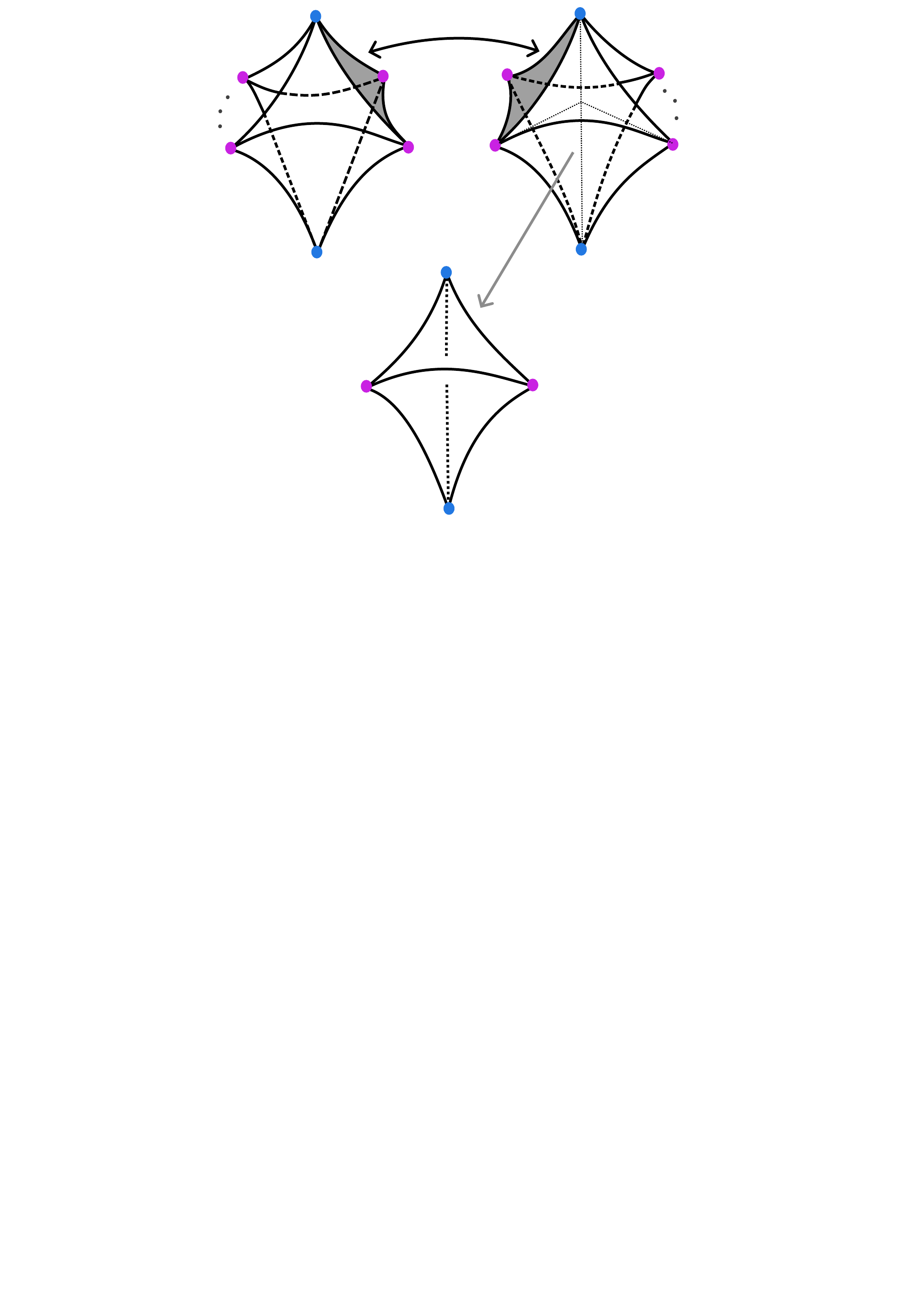}
        \small
        \put(45,9){F}
        \put(55,9){E}
        \put(52,16){D}
        \put(51,24){A}
        \put(44,24){B}
        \put(56,24){C}
        
        \end{overpic}
        \caption{Above: two symmetric bipyramids with the gluing along their faces indicated by an arrow. Ultra-ideal vertices are blue and ideal vertices are pink. Below: a `wedge' with edges labelled.}
        \label{fig:wedges}
    \end{figure}
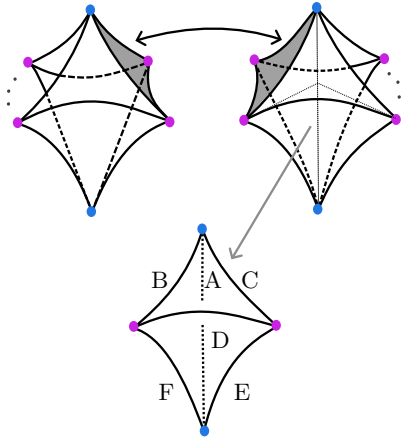
    We have assumed that there are two $n$-gons and two $m$-gons about each vertex. This implies that for $\alpha_n$ the interior angle of the $n$-gon and $\alpha_m$ the interior angle of the $m$-gon, $2\alpha_n+2\alpha_m=2 \pi$ so $\alpha_n+\alpha_m= \pi$. Slice the bipyramids from \cite{ACM} along the checkerboard surfaces which separate each bipyramid into two pyramids. We can then see that $P^+$ and $P^-$ will be the result of gluing the resulting pyramids on their faces. This gluing matches the corresponding $D$ edges with dihedral angles $(\pi - \alpha_n)/2$ and $(\pi- \alpha_m)/2$. Using that $(\pi - \alpha_n + \pi - \alpha_m)/2= \pi /2$, we see that this gluing forms right-angled generalized checkerboard polyhedra. Hence $L$ is right-angled generalized completely realizable.

\end{proof}

\begin{corollary}
\label{TorusCor}
For projection surface $F = T^2$, $L$ is RGCR if and only if it is the triaxial link or the square weave.
\end{corollary}

\begin{proof}
Theorems \ref{TFAE} and \ref{RegularityThm} imply that any RGCR link with alternating projection on the torus is semi-regular. We can then apply \cite[Theorem 5.1]{CKPBiperiodic}.  
\end{proof}

Theorem \ref{TFAE} has implications for commensurability. 

\begin{prop}
Consider $RGCR$ links where $F$ has genus greater than 1. All $RGCR$ links with the same $n$ and $m$-gon faces in their checkerboard surfaces are commensurable.
\end{prop}

\begin{proof}
 By Theorem \ref{TFAE} any two RGCR links, $L_1$ and $L_2$, with the same $n$ and $m$-sided polygons correspond to regular $[n, m, n, m]$ tilings of $\mathbb{H}^2$. The group of isometries of this tiling is the collection of hyperbolic isometries generated by reflecting in the three sides of a triangle with angles $\pi/n$, $\pi/m$, and $\pi/2$. We can see this by noting that these isometries preserve the $[n, m, n, m]$ tiling and that any isometry of the tiling is the result of a combination of these reflections, as shown for tilings consisting of one type of polygon in \cite{EEK}. 
 
 Doubling this triangle along the edge between the $\pi/n$ and $\pi/m$ angles gives us a quadrilateral in $\pi(L)$ which corresponds to two bipyramid wedges sharing edge $D$ in the hyperbolic structure of each link's complement (see the labelling in Figure \ref{fig:wedges}).
 Reflections in the faces of this pair of wedges are then isometries of the three-dimensional tiling of $\mathbb{H}^3$ by the bipyramids corresponding to the $[n, m, n, m]$ tiling. The complement of $L_1$ in $F_1 \times I$, $M_1$, and the complement of $L_2$ in $F_2 \times I$, $M_2$, both finitely cover the orbifold which is the quotient of the three-dimensional tiling by these these isometries. Take the intersection of $\pi_1(M_1)$ and $\pi_1(M_2)$. The result is a group $G$ corresponding to a finite cover of $M_1$ and $M_2$.
 
\end{proof}

\begin{remark}
  \normalfont For semi-regular links on the torus, Champanerkar, Kofman, and Purcell related the commensurability of the links to the polygons in their corresponding tilings \cite[Theorem 4.1]{CKPBiperiodic}. 
\end{remark}

\section{Classifying Right-angled Generalized Completely Realizable Links}

It is rare for alternating links in $S^3$ to be right-angled and completely realizable. In particular, Gan showed that there are only three such links \cite{Gan}. Moving to the thickened torus, we can see from Corollary \ref{TorusCor} that there are only two tilings that correspond to RGCR links on the torus. In this section find an upper bound on the number of RGCR links in each genus $g\geq 2$ thickened surface and provide a table of the types of tilings that correspond to RGCR links in $F \times I$ with genus 2 to 4. 
 
 We next consider right-angled knots. Champanerkar, Kofman, and Purcell conjecture that there are no links in $S^3$ with right-angled structure \cite{CKPright}. Supporting this conjecture, Gan's work proves that there are no right-angled and completely realizable knots in $S^3$. We show that this conjecture does not extend to links in thickened surfaces by constructing an example of an RGCR knot.

\subsection{Counting RGCR Links}
    In order to find an upper bound on the number of RGCR links we need to both find the number of tilings corresponding to RGCR links for each genus $g$ projection surface and see how many RGCR links correspond to a given number of tiles in a given tiling. As shown in Figure \ref{fig:2linkssametile} distinct links can correspond to the same tiling and number of tiles. 

    \begin{example}
    \normalfont Figure \ref{fig:2linkssametile} provides an example of two distinct links on a genus 3 surface which derive from the same tiling of $\mathbb{H}^2$ by octagons and contain the same number of octagons in their projection diagrams. One can check that the links are distinct by noting that they have different numbers of components. 
    
    \begin{figure}
        \centering
        \includegraphics[trim= 0.25in 0.5in 0.25in 1.5in, width=4in]{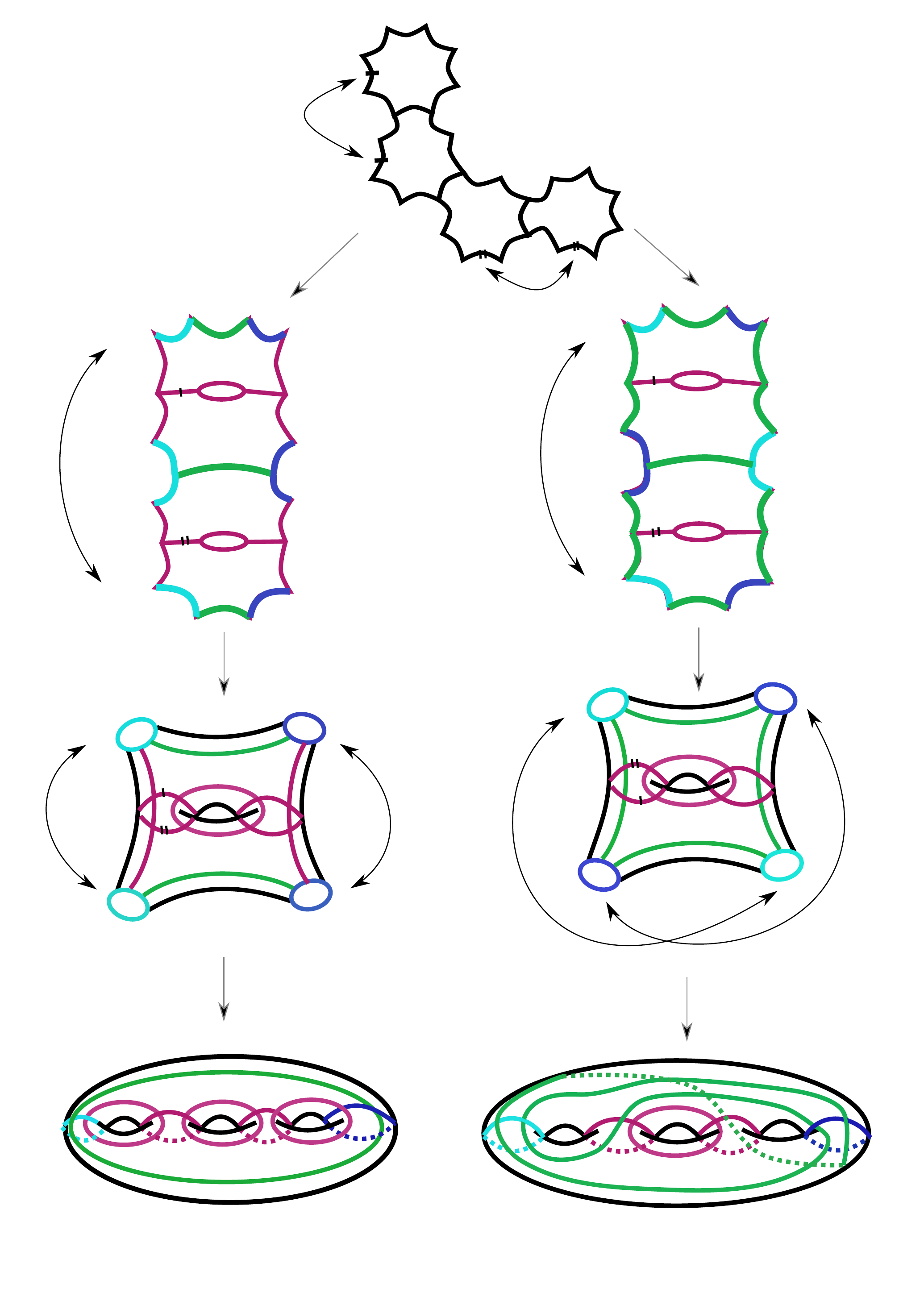}
        \caption{Two links from a tiling of $\mathbb{H}^2$ by right-angled octagons. Both links have $m=8=n$ and $k_n=2=k_m$. We begin with a fundamental domain of four right-angled octagons and proceed to identify edges to form the projection diagrams of the links. Double headed arrows indicate gluing while single headed arrows indicate the results of the gluings. The different choices of gluing result in distinct links.}
        \label{fig:2linkssametile}
    \end{figure}
    \end{example}
    
    \begin{defn}
    \normalfont Denote the number of times that $n$-gons occur in a checkerboard surface of a link diagram by $k_n$ and the number of times that $m$-gons occur in the other checkerboard surface by $k_m$.
    \end{defn} 
    
    Let $\alpha_n$ be the interior angle of each $n$-gon and $\alpha_m$ be the interior angle of each $m$-gon. By Theorem \ref{TFAE}, an RGCR link must have $\alpha_n + \alpha_m= \pi$. Additionally, each $n$-gon shares each of its edges with an $m$-gon and each $m$-gon corresponds to $m$ $n$-gons. So, $k_n n=k_mm$. Note that when $n=m$ we have $k_n=k_m$.
    
    \begin{prop}[Gauss-Bonnet]
    
    Let $F$ be the projection surface of a link $L$ with polygons $T_n$ and $T_m$ in its diagram. Then
    $-2\pi \chi (F) = a(F) = k_n a(T_n) + k_m a(T_m),$ where $a$ denotes area.
    \end{prop}
  
    \begin{theorem}
    \label{finitelymany}
    Given a projection surface $F$ of genus $g$, there are finitely many tilings of $\mathbb{H}^2$ and $\mathbb{E}^2$ which correspond to RGCR links, $L$, with $\pi(L)$ on $F$. Moreover, for each $F$ of genus $g > 1$, the number of RGCR links is bounded above by $\displaystyle{\bigg(\frac{310g^2}{9}-\frac{101g}{3}+4\bigg)[(84g-83)!]}$.
    \end{theorem}
    
     \begin{proof}

    If $F$ is a torus, then by Corollary \ref{TorusCor} and \cite[Theorem 5.1]{CKPBiperiodic} the only Euclidean tilings which RGCR links correspond to are the tiling by regular hexagons and triangles and the tiling by squares. 

    Next consider when $F$ has $g>1$.
     First, we will bound the number of fundamental domains of tilings that correspond to RGCR links for a given $F$ with fixed genus $g$. In other words, we will bound the number of tuples $(m, n, k_m, k_n)$ which correspond to RGCR links in terms of $g$. 
      
     Without loss of generality, assume that $n \leq m$. We have that $n\geq 3$ because the links do not have bigons, so $m \geq 5$ because we are assuming that the $[n,m,n,m]$ tilings are hyperbolic. Using Gauss--Bonnet we find,
    \begin{align*}
        -2 \pi (2-2g) = -2 \pi \chi(F) & = a(F)\\
       & = \sum^{k_n}a(T_n)+ \sum^{k_m}a(T_m)
        \\
        &=k_n(n(\pi -\alpha_n)-2\pi)+ k_m(m(\pi - \alpha_m) -2\pi)
        \\
        & =k_n(n(\pi-\alpha_n)-2\pi) + \left(\frac{k_n n}{m}(m\alpha_n -2\pi)\right).
    \end{align*}
    
    \noindent Simplifying we find
    \begin{equation} \label{combarea_eq1}
    4 (g-1)= k_n \left(n-2- \frac{2n}{m}\right),
    \end{equation}
    and
    \begin{equation}
    \label{combarea_eq2}
        4(g-1)= k_m\left(m-2 -\frac{2m}{n}\right).
    \end{equation}
    
    \noindent Equations \eqref{combarea_eq1} and \eqref{combarea_eq2} imply that $k_n$ and $k_m$ are determined by $n$, $m$, and $g$. Therefore it suffices for us to find an upper bound on the number of pairs $(m, n)$ in terms of $g$. 
    
     If $k_n=1$, then $n=m=4g$ by \eqref{combarea_eq1} because our assumption that $n \leq m$ implies that $k_n \geq k_m$. Therefore there is only one option for $n$ and $m$ when $k_n=1$. 
     
     When $k_n \geq 2$,
    $$n-2 - \frac{2n}{m}= \frac{4(g-1)}{k_n} \leq \frac{4(g-1)}{2}=2g-2,$$
    
    \noindent simplifying, we find that
    $$ n\left(\frac{m-2}{m}\right) \leq 2(g-1)+2=2g.$$
    
    \noindent Recall that $m \geq 5$, so $\displaystyle{\frac{3}{5}\leq \frac{m-2}{m}}$. Therefore, 
    $$n \left(\frac{3}{5}\right) \leq n\left(\frac{m-2}{m}\right) \leq 2g,$$
    
   \noindent which implies that $\displaystyle{n \leq \frac{10g}{3}}.$ 
    
    We bound $m$ by starting with Equation \eqref{combarea_eq2}. The same set of inequalities hold except with $n$ replaced with $m$, $n\geq 3$, and $k_m \geq 1$. Therefore,
     $$m \left(\frac{1}{3}\right) \leq m\left(\frac{n-2}{n}\right) \leq \frac{4(g-1)}{1}+2 = 4g-2.$$
     
    \noindent Thus, $\displaystyle{m \leq 12g-6}.$
    
   Our goal is to bound the number of possible pairs $(m,n)$ for each fixed genus $g$. Using that $n \geq 3$ and that there is only the one option for $(m, n)$ with a value of $n$ larger than $10g/3$ (the case when $n=m$ and $k_n=1=k_m$), there are up to $\displaystyle{\frac{10g}{3}-2+1=\frac{10g-3}{3}}$ choices for $n$. Using that $m \geq 5$, there are \newline$12g-6-4=12g-10$ choices for $m$. 
   
   For each $\displaystyle{m \leq \frac{10g}{3}}$ there are $m-2$ choices for what $n$ can be and for each $m \geq (10g)/3$ there are $\displaystyle{\frac{10g-3}{3}}$ possibilities for $n$. Therefore the number of pairs $(m, n)$ is bounded above by: 
   
    $\displaystyle{\bigg[\sum_{m=5}^{(10g)/3}(m-2)\bigg] + ((10g-3)/3)(12g-6-((10g)/3))}$

    $=\displaystyle{\bigg[\sum_{m=5}^{(10g)/3}(m-2)\bigg] + (\frac{2}{9} (10g-3)(13g-9))}$

    $\displaystyle{=\frac{310g^2}{9}-\frac{101g}{3}+4}$.
    
    Distinct links can have the same tuple $(m, n, k_m, k_n)$ (see Figure \ref{fig:2linkssametile}). So, we also need to bound the number of RGCR links that correspond to the same tiling $T$. 
    The orientation preserving symmetry group of $T$ is finitely generated by two rotations preserving the checkerboard coloring. This group of symmetries is the orientation preserving subgroup of the triangle group which is all of the symmetries of $T$. The RGCR links corresponding to this tiling are then the result of taking the quotient of $\mathbb{H}^2$ by surface subgroups of this symmetry group of the same index. 
    By the $84(g-1)$ Theorem this index $x$ satisfies, $x \leq 84(g-1)$ (see \cite{FB}). The number of subgroups of index $x$ of a finitely generated group $G$ is bounded above by $x \cdot x!^{d(G)-1}$, where $d(G)$ is the minimal number of generators of $G$ \cite{SubgroupGrowth}. Therefore the number of weakly generalized alternating links on $F$ that correspond to $T$ is bounded by $x \cdot x!^{d(G)-1}= x \cdot x! \leq (x+1)! \leq (84g-83)!.$
    \end{proof}
    
  An alternative method for finding a bound on the number of links corresponding to a given tiling with a given number of polygons in their diagrams is to consider the number of feasible gluings of each of the faces of $\pi(L)$ along their edges. However, this results in a larger upper bound.
  
  \begin{remark}
    \normalfont We cannot bound the number of polygons that appear in $\pi(L)$ for an RGCR link on the torus with these calculations. For a torus, equation \eqref{combarea_eq1} becomes $0=k_n(n-2-\frac{2n}{m})$ with $n=4=m$ or $n=3$ and $m=6$.
  \end{remark}

    \subsection{A Right-angled Knot}
    Gan showed that among alternating links in $S^3$ there are no right-angled completely realizable knots \cite[Theorem 3.14]{Gan}. More broadly, Champanerkar, Kofman, and Purcell conjecture that there does not exist a right-angled knot in $S^3$ \cite[Conjecture 5.12]{CKPright}. Here, we show that this conjecture does not extend to links in thickened surfaces. 

\begin{theorem}
\label{KnotThm}
There exists a right-angled knot in thickened surfaces.
\end{theorem}
\begin{proof}
    We construct a right-angled knot in a thickened genus $2$ surface in Figure \ref{RGCR Knot}. This knot corresponds to a hyperbolic tiling by regular octagons.
\end{proof}
\begin{figure}
    \centering
    \includegraphics[width=5.75in]{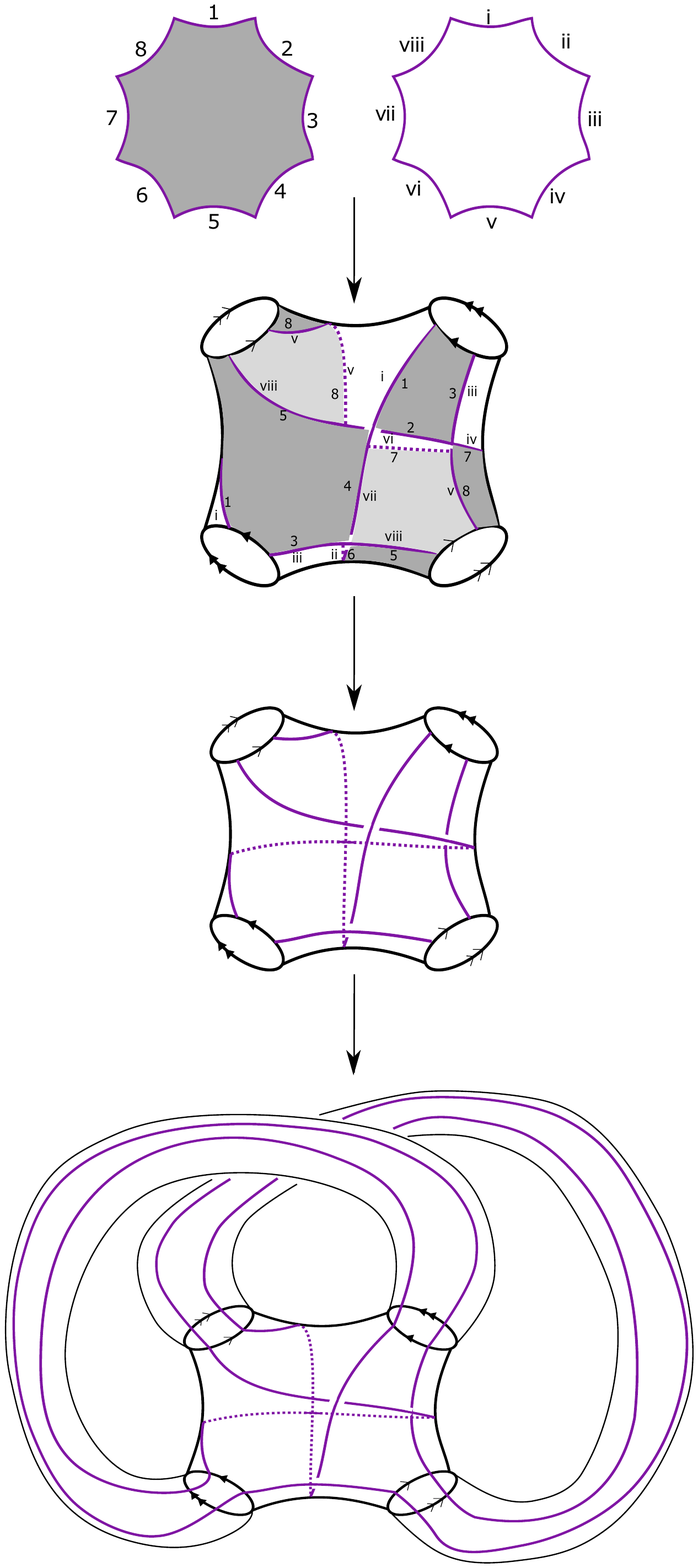}
    \caption[A right-angled knot in $F \times I$.]{An RGCR knot.}
    \label{RGCR Knot}
\end{figure}

   \newpage
\begin{center}
\begin{longtable}{|c|c|c|c|c|}

\caption{The tilings, and quantities of diagram faces, which correspond to RGCR links on projection surfaces of genus $2$ to genus $4$. In other words, the possible values of $m$, $n$, $k_m$, and $k_n$ for each projection surface. These values are calculated from Equations \ref{combarea_eq1} and \ref{combarea_eq2} and the bounds in the proof of Theorem \ref{finitelymany}. For further detail on RGCR links corresponding to higher genus surfaces and calculations of the interior angles of their polygons see \cite{RKKThesis}.}

\label{RGCR_tilings_table} \\

\hline \multicolumn{1}{|c|}{\textbf{Genus}} & \multicolumn{1}{c|}{\textbf{m}} & \multicolumn{1}{c|}{\textbf{n}} & \multicolumn{1}{c|}{$\mathbf{k_m}$}& \multicolumn{1}{c|}{$\mathbf{k_n}$}\\ \hline 
\endfirsthead

\multicolumn{5}{c}%
{{\bfseries \tablename\ \thetable{} -- continued from previous page}} \\
\hline \multicolumn{1}{|c|}{\textbf{\text{    }Genus\text{ 
   }}} & \multicolumn{1}{c|}{\textbf{\text{ }\text{    }m\text{ }\text{ }}} & \multicolumn{1}{c|}{\textbf{\text{ }\text{ }n \text{ }\text{ }}} & \multicolumn{1}{c|}{$\mathbf{\text{ }\text{ }k_m \text{ }\text{ }}$}& \multicolumn{1}{c|}{$\mathbf{k_n}$} \\ \hline 
\endhead

\hline \hline
\endlastfoot

        2 & 5 & 4 & 8 & 10  \\ \hline
        2 & 6 & 4 & 4 & 6 \\ \hline
        2 & 7 & 3 & 12 & 28 \\ \hline
        2 & 8 & 3 & 6 & 16 \\ \hline
        2 & 8 & 4 & 2 & 4 \\ \hline
        2 & 9 & 3 & 4 & 12 \\ \hline
        2 & 10 & 3 & 3 & 10 \\ \hline
        2 & 10 & 5 & 1 & 2 \\ \hline
        2 & 12 & 3 & 2 & 8 \\ \hline
        2 & 12 & 4 & 1 & 3 \\ \hline
        2 & 18 & 3 & 1 & 6 \\ \hline
        2 & 8 & 8 & 1 & 1  \\ \hline
        2 & 6 & 6 & 2 & 2 \\ \hline
        2 & 5 & 5 & 4 & 4 \\ \noalign{\global\arrayrulewidth2pt}
 \hline
 \noalign{\global\arrayrulewidth0.4pt}
        3 & 5 & 4 & 16 & 20  \\ \hline
        3 & 6 & 4 & 8 & 12  \\ \hline
        3 & 6 & 5 & 5 & 6  \\ \hline
        3 & 7 & 3 & 24 & 56  \\ \hline
        3 & 8 & 3 & 12 & 32  \\ \hline
        3 & 8 & 4 & 4 & 8 \\ \hline
        3 & 9 & 3 & 8 & 24  \\ \hline
        3 & 9 & 6 & 2 & 3 \\ \hline
        3 & 10 & 3 & 6 & 20  \\ \hline
        3 & 10 & 5 & 2 & 4  \\ \hline
        3 & 12 & 3 & 4 & 16  \\ \hline
        3 & 12 & 4 & 2 & 6  \\ \hline
        3 & 14 & 3 & 3 & 14 \\ \hline
        3 & 14 & 7 & 1 & 2  \\ \hline
        3 & 18 & 3 & 2 & 12  \\ \hline
        3 & 20 & 4 & 1 & 5  \\ \hline
        3 & 30 & 3 & 1 & 10  \\ \hline
        3 & 12 & 12 & 1 & 1  \\ \hline
        3 & 8 & 8 & 2 & 2  \\ \hline
        3 & 6 & 6 & 4 & 4  \\ \hline
        3 & 5 & 5 & 8 & 8  \\ 
        \noalign{\global\arrayrulewidth2pt}
 \hline
 \noalign{\global\arrayrulewidth0.4pt}
        4 & 5 & 4 & 24 & 30 \\ \hline
        4 & 6 & 4 & 12 & 18  \\ \hline
        4 & 7 & 3 & 36 & 84  \\ \hline
        4 & 7 & 4 & 8 & 14  \\ \hline
        4 & 8 & 3 & 18 & 48  \\ \hline
        4 & 8 & 4 & 6 & 12  \\ \hline
        4 & 9 & 3 & 12 & 36  \\ \hline
        4 & 10 & 3 & 9 & 30  \\ \hline
        4 & 10 & 4 & 4 & 10  \\ \hline
        4 & 10 & 5 & 3 & 6  \\ \hline
        4 & 12 & 3 & 6 & 24  \\ \hline
        4 & 12 & 4 & 3 & 9  \\ \hline
        4 & 12 & 6 & 2 & 4  \\ \hline
        4 & 15 & 3 & 4 & 20  \\ \hline
        4 & 16 & 4 & 2 & 8 \\ \hline
        4 & 18 & 3 & 3 & 18  \\ \hline
        4 & 18 & 9 & 1 & 2  \\ \hline
        4 & 24 & 3 & 2 & 16 \\ \hline
        4 & 28 & 4 & 1 & 7  \\ \hline
        4 & 42 & 3 & 1 & 14  \\ \hline
        4 & 16 & 16 & 1 & 1  \\ \hline
        4 & 10 & 10 & 2 & 2  \\ \hline
        4 & 8 & 8 & 3 & 3  \\ \hline
        4 & 7 & 7 & 4 & 4  \\ \hline
        4 & 6 & 6 & 6 & 6  \\ \hline
        4 & 5 & 5 & 12 & 12  \\ 
        \noalign{\global\arrayrulewidth2pt}

\end{longtable}
\end{center}
    \newpage
    
    \printbibliography

@article{AdamsAlone,
 author= "Colin Adams",
 title= "Toroidally alternating knots and links",
 journal= "Topology",
 volume= "33",
 number= "2",
 pages= "353-369",
 year= "1994"
 }

@article{ACM,
  title= "Generalized bipyramids and hyperbolic volumes of alternating k-uniform tiling links",
  author= "Colin Adams and Aaron Calderon and Nathaniel Mayer",
  journal= "Topology and its Applications",
  volume= "271",
  pages= "1-28",
  year= "2020"
  }

@article{Adams.etal,
  title= "Hyperbolicity of links in thickened surfaces",
  author= "Colin Adams and Carlos Albors-Riera and Beatrix Haddock and Zhiqi Li and Daishiro Nishida and Braeden Reinoso and Luya Wang",
  journal= "Topology and its Applications",
  volume= "256",
  number= "1",
  pages= "262-278",
  year= "2019"
  }

@article{harlequin,
 title= "Cusp size bounds from singular surfaces in hyperbolic 3-manifolds",
 author= "Colin Adams and Adam A. Colestock and James Fowler and William D. Gillam and E. Katerman",
 journal= "Transactions of the American Mathematical Society",
 volume= "358",
 number= "2",
 pages= "727–741",
 year= "2005"
 }

@article{AR,
  title={On Archimedean link complements},
  author={Iain R. Aitchison and Lawrence Reeves},
  journal={Journal of Knot Theory and Its Ramifications},
  year={2002},
  volume={11},
  pages={833-868}
}

@article{CKPBiperiodic,
   title={Geometry of biperiodic alternating links},
   volume={99},
   number={3},
   journal={Journal of the London Mathematical Society},
   publisher={Wiley},
   author={Champanerkar, Abhijit and Kofman, Ilya and Purcell, Jessica S.},
   year={2018},
   pages={807–830} }

@article{CKPright,
author = {Abhijit Champanerkar and Ilya Kofman and Jessica S Purcell},
title = {{Right-angled polyhedra and alternating links}},
volume = {22},
journal = {Algebraic \& Geometric Topology},
number = {2},
publisher = {MSP},
pages = {739 -- 784},
year = {2022}
}

@article{DattaandGupta,
  title={Semi-regular Tilings of the Hyperbolic Plane},
  author={Basudeb Datta and Subhojoy Gupta},
  journal={Discrete \& Computational Geometry},
  year={2021},
  volume={65},
  pages={531-553}
}

@article{DattaandMaity,
author = {Datta, Basudeb and Maity, Dipendu},
year = {2018},
month = {12},
pages = {3296-3309},
title = {Semi-equivelar maps on the torus and the Klein bottle are Archimedean},
volume = {341},
journal = {Discrete Mathematics},
}

@article{Gan,
  title={Alternating links with totally geodesic checkerboard surfaces},
  author={Hong-Chuan Gan},
  journal={Algebraic \& Geometric Topology},
  year={2021}
}

@article{GreeneAlt,
  title={Alternating links and definite surfaces},
  author={Joshua Evan Greene},
  journal={Duke Mathematical Journal},
  year={2017},
  volume={166},
  pages={2133-2151}
}

@article{Hayashi,
  title={Links with alternating diagrams on closed surfaces of positive genus},
  author={Chuichiro Hayashi},
  journal= {Mathematical Proceedings of the Cambridge Philosophical Society},
  year={1995}
}

@article{HowieAlt,
  title={A characterisation of alternating knot exteriors},
  author={Joshua A. Howie},
  journal={Geometry \& Topology},
  year={2015}
}

@phdthesis{HowieThesis,
    title    = {Surface-alternating knots and links},
    school   = {University of Melbourne},
    author   = {Joshua Howie},
    year     = {2015}
}

@article{HowiePurcell,
 title= {Geometry of alternating links on surfaces},
 author= {Joshua Howie and Jessica Purcell},
 journal= {Transactions of the American Mathematical Society},
 volume= "373",
 pages= "2349-2397",
 year= "2020"
 }

@phdthesis{RKKThesis,
    title    = {Right-angled links in thickened surfaces},
    school   = {Temple University},
    author   = {Rose Kaplan-Kelly},
    year     = {2023}
}

@book{SubgroupGrowth,
  author = {Alexander Lubotzky and Dan Segal},
  year = {2003},
  title = {Subgroup Growth},
  publisher = {Birkh\"auser}
 }

@article{MenascoClassic,
  title={Closed incompressible surfaces in alternating knot and link complements},
  author={William Menasco},
  journal={Topology},
  year={1984},
  volume={23},
  pages={37-44}
}

@article{Menasco,
 title={Polyhedra representation of link complements},
 author= {William Menasco},
 journal={Low Dimensional Topology},
 year= {1983},
 volume= {20},
 pages={305-325}
 }

@article{Ozawa,
  title={Non-triviality of generalized alternating knots},
  author={Makoto Ozawa},
  journal={Journal of Knot Theory and Its Ramifications},
  year={2005},
  volume={15},
  pages={351-360}
}

@article{Schlenker,
  title={Hyperbolic manifolds with polyhedral boundary},
  author={Jean-Marc Schlenker},
  journal={arXiv: Geometric Topology},
  year={2002}
}

@article{EEK,
 title={Regular Tessellations of Surfaces and (p, q, 2)-Triangle Groups},
 author= {Allan L. Edmonds and John H. Ewing and Ravi S. Kulkarni},
 journal= {Annals of Mathematics},
 volume= {116},
 number= {1},
 year= {1982},
 pages= {113-132}
 }

@book{FB,
 title= {A Primer on Mapping Class Groups},
 author= {Benson Farb and Dan Margalit},
 year= {2012},
 publisher= {Princeton University Press}
 }

@article{TT,
author = {Morwen Thistlethwaite and Anastasiia Tsvietkova},
title = {{An alternative approach to hyperbolic structures on link complements}},
volume = {14},
journal = {Algebraic \& Geometric Topology},
number = {3},
publisher = {MSP},
pages = {1307 -- 1337},
year = {2014}
}
\end{document}